\documentclass[a4paper,10pt]{amsart}
\usepackage[latin1]{inputenc}
\usepackage[english]{babel}
\usepackage{amsmath,amsthm,amsfonts,amssymb,amstext,mathrsfs} 
\usepackage{hyperref} 
\usepackage{mathabx}
\usepackage{cite}
\usepackage{caption}
\usepackage{subcaption}
\usepackage{accents}
\usepackage{graphicx}
\usepackage{todonotes}

\newtheorem{theorem}{Theorem}[section]

\newtheorem{lemma}[theorem]{Lemma}

\newtheorem{corollary}[theorem]{Corollary}
\theoremstyle{definition}
\newtheorem*{assumption}{Assumption}
\newtheorem{definition}[theorem]{Definition}
\theoremstyle{remark}
\newtheorem{remark}{Remark}

\newtheorem{example}[theorem]{Example}

\makeatletter
\let\S\@undefined
\makeatother

\DeclareMathOperator{\rank}{rank}

\DeclareMathOperator{\conv}{conv}

\def\BB{\mathscr{B}}
\def\LL{\mathscr L}
\def\FF{\mathscr F}

\usepackage{tikz}
\usetikzlibrary{matrix,arrows,patterns}
\usetikzlibrary{arrows,shapes,automata,backgrounds,petri}
\tikzstyle{knoten}=[fill,shape=circle,inner sep=2pt,outer sep=0pt,minimum size=2pt]
\tikzstyle{background}=[rectangle,
                                                fill=gray!10,
                                                inner sep=0.2cm,
                                                rounded corners=5mm]
\title{Bergman complexes of lattice path matroids}
\author{Emanuele Delucchi, Martin Dlugosch}
\address[Emanuele Delucchi]{Departement de Math\'ematiques, Universit\'e de Fribourg, Chemin du Mus\'ee 21, 1700 Fribourg, Switzerland.}
\email[Emanuele Delucchi]{emanuele.delucchi@unifr.ch}
\address[Martin Dlugosch]{Fachbereich Mathematik und Informatik, Universit\"at Bremen,
  Bibliothekstra\ss{}e 1, 28359 Bremen, Bremen, Germany.}
\email[Martin Dlugosch]{mdlug@math.uni-bremen.de }

\begin{document}
\maketitle

\vspace{-10pt}

\begin{abstract}
We give an explicit description of the poset of cells of Bergman
complexes of lattice path matroids and establish a criterion for its simpliciality, in terms of the shape of the
bounding paths.
\end{abstract}

\vspace{-5pt}

\section{Introduction}

The term {\em Bergman fan} has come to denote the polyhedral fan
given as the logarithmic limit set of a complex algebraic subvariety
$V$ of $\mathbb C^n$. Logarithmic limit sets were first
introduced by George M.\ Bergman in \cite{BE71}, and their fan
structure was described by Bieri and Groves \cite{BiGr}. Sturmfels \cite{STU02}
made the key observation that, if $V$ is defined by linear equations,
its Bergman fan can be constructed from the {matroid} associated to
$V$. In this case, the Bergman fan is the tropicalization of the complement variety of the hyperplane arrangement  determined in $V$ by the intersection of the standard coordinate hyperplanes of $\mathbb C^n$. This leads to the very important role played by Bergman fans in tropical geometry (we point the interested reader to the book \cite{McSt} for more on this subject). As Sturmfels' construction can be carried out also for
nonrepresentable matroids, we get a Bergman fan -- indeed, a tropical variety -- associated to every
matroid $M$. There is no loss of information in considering, instead
of the whole fan, just its intersection with the unit sphere: the resulting
spherical complex $\Gamma_M$ is called the {\em Bergman complex} of the given
matroid $M$ (see Section \ref{sec:bergmandef} for the precise definitions).
 The homotopy type of the complex $\Gamma_M$
coincides with that of the order
 complex of $\LL(M)$, the lattice of flats of $M$ (see Section \ref{def:matroidS}). This has been established by
 Ardila and Klivans \cite{AK}, who proved that in fact the order
 complex of $\LL(M)$ subdivides $\Gamma_M$. This raises the
 question of the polyhedral structure of $\Gamma_M$. Here,  progress
 was made by
 Feichtner and Sturmfels \cite{FS04}, who proved that $\Gamma_M$ is
 subdivided by the nested set complex of $M$, a much coarser
 (simplicial) complex than the order complex of $\LL(M)$. As an additional improvement, the second author \cite{Dl11} described a decomposition of the
 matroid types (see Section \ref{def:matroidpoly}) associated to faces
 of the Bergman complex into connected direct summands.   Recently,
 Rinc\'on \cite{Rincon} found and implemented an algorithm to
efficiently compute a certain simplicial subdivision of the nested set
complex, called cyclic Bergman complex.

From a combinatorial point of view, an open problem is to find an
explicit description of the face structure of the Bergman complex
$\Gamma_M$, improving on the aforementioned results which describe
simplicial complexes subdividing $\Gamma_M$. A more general question, of
interest for computations as well as for the structure theory,
and which can't be effectively
answered in general either, is 
whether the Bergman complex of a given matroid is itself simplicial. In this paper we address both the characterization of simpliciality as well as the explicit face structure of Bergman complexes for a special class of matroids.

{\em Lattice path matroids} were introduced by Bonin, de Mier and Noy
\cite{BdMN} as a family of transversal
matroids whose bases can be characterized by means of the
 lattice paths contained in the region of the plane bounded by two given
lattice paths. This class of matroids enjoys a host of nice
enumerative and
structural properties \cite{BdM06}, and can be characterized
among all matroids by a list of excluded minors \cite[Theorem 3.1]{BdM09}. 
Their rich structure theory 
allowed lattice path matroids to attract attention both as a fertile setting for significative 
stepping stones toward general results (e.g., Stanley's $M$-vector conjecture \cite{SJ} or the theory hyperplane splits of matroid polytopes \cite{ChRa}) as well as  by offering a convenient general framework for the study of some subclasses of independent interest, such as generalized Catalan matroids \cite[Section 4]{BdM06}.
In the context of the study of  Grassmannians, where generalized Catalan matroids are called Schubert matroids, lattice path matroids correspond precisely to `Richardson matroids', a special case of the positroids used as indices of cells in Postnikov's stratification of the
totally nonnegative Grassmannian \cite{Post} (see e.g.\ \cite[Introduction]{KLS} for an overview of this subject) -- for these `special
cells' the computation of the corresponding `Grassmann necklace' is
particularly tractable \cite[Section 6]{OH}.\\

In this paper we determine the polyhedral structure of the Bergman
complex of a given lattice path matroid and formulate a necessary and
sufficient condition for $\Gamma_M$ to be simplicial. In the geometric
spirit of lattice path matroids, our characterizations are in terms of
the shape of the bounding paths.

The structure of the paper is as follows. First, in Section \ref{sec:def} we review some basic notions 
and we derive a
description of faces and vertices of the Bergman complex of a lattice path
matroid in terms of {`bays'} and {`land necks'} of the bounding
paths (Lemma \ref{rem:strait}, see also Remark \ref{interpret}).

Section \ref{sec:simpl} contains our first main result, Theorem
\ref{maintheo1}, where, in terms of bays and land necks, we characterize simpliciality of faces of the Bergman
complex and, as a corollary, we characterize those lattice path matroids that have a simplicial Bergman
complex (Corollary \ref{cor:iffs}). 

In Section \ref{sec:faces} we introduce a poset
(again defined in terms of bays and (non-)land necks,
see Definition \ref{def:Q}) which we prove 
 to be isomorphic to the face poset of
the Bergman complex (Theorem \ref{maintheo2}). We close with an explicit expression for the polyhedral structure of faces of the Bergman complex (Corollary \ref{cor:poly}).\\

\noindent{\em Acknowledgments.}  We thank Anna de Mier for helpful
discussions, and the anonymous referees for valuable suggestions leading to substantial improvement of the paper.
Emanuele Delucchi has been partially supported by the Swiss National Science
Foundation professorship grant PP00P2\_150552/1.

\section{Preliminaries}\label{sec:def}

\subsection{Matroids}\label{def:matroidS}
We sketch some of the basics of matroid theory, in order to provide the basic 
definitions and to set some notation. For a 
thorough introduction to the subject and as a standard reference we point to \cite{Ox11}.

\begin{definition}\label{def:matroid}
  Let $E$ denote a finite set, $\mathscr P(E)$ the set of its subsets. 
  A nonempty family $\BB \subseteq \mathscr P(E)$ is the
  set of bases of a matroid on the ground set $E$ if, given
  $B_1,B_2\in \BB$ and $e\in B_1\setminus B_2$, there is $f\in
  B_2\setminus B_1$ such that $(B_1\setminus \{e\}) \cup \{f\} \in
  \BB$. A {\em matroid} can be given as the pair $(E,\BB)$. Given a
  matroid $M$, we will denote by $\BB(M)$ its set of bases. 

A {\em loop} of $M$ is any $e\in E$ that is not contained in any
$B\in \BB(M)$. A {\em coloop} of $M$ is any $e\in E$ that is
contained in every $B\in \BB(M)$.
\end{definition}

\def\rk{\operatorname{rk}}
\begin{remark}\label{def:CD} 

If $M$ is a matroid on the ground set $E$, then all the elements of
  $\BB(M)$ have the same cardinality, which is called the {rank}
  of $M$.
  More generally, given a subset $A\subseteq E$, define the {\em rank} of $A$ to be
$$
\rk(A) := \max \{\vert A\cap B \vert : B\in \BB(M)\}.
$$
Consider the families
  \begin{eqnarray*}
    \BB(M)[A]:= \{B\cap A : B\in \BB(M), \vert A\cap B\vert = \rk(A) \}, \\
    \BB(M)/A:= \{B \setminus  A : B \in \BB(M), \vert A\cap B\vert = \rk(A)\}.
  \end{eqnarray*}
%
These satisfy Definition \ref{def:matroid} and thus describe  
a matroid $M[A]$ on the ground set $A$ and a matroid $M/A$ on the
ground set $E\setminus A$, respectively. 
The matroid $M[A]$ is called
the {\em restriction} of $M$ to $A$, while $M/A$ is the {\em
  contraction} of $A$ in $M$. A matroid that is obtained from $M$ by a
sequence of contractions and restrictions is called a {\em minor} of
$M$.
Notice that for every $A\subseteq E$ the rank of $M[A]$ is
$\rk(A)$. In particular, the rank of $M$ is $\rk(E)$.
\end{remark}

\begin{remark}\label{rem:CoCoCo} 
  Let $M_1$ and $M_2$ be matroids on disjoint ground sets $E_1$,
  $E_2$. Their {\em direct sum} is the matroid $M_1 \oplus M_2$ on the
  ground set $E_1\cup E_2$ with
  set of bases $$\BB(M_1\oplus M_2) =\{B_1\cup B_2 : B_1\in \BB(M_1),\,
  B_2\in \BB(M_2)\}.$$
  A matroid $M$ on the ground set $E$ is {\em connected} if there is
  no nontrivial partition $E=E_1\uplus E_2$ with $M=M[E_1]\oplus
  M[E_2]$. Let $E=E_1\uplus \cdots \uplus E_c$ be a 
  partition of $E$ such that $M= 
  M[E_1]\oplus \ldots \oplus M[E_c]
  $ and $M[E_j] $ is
  connected for every $j=1,\ldots, c$. Then the $M[E_j]$ are uniquely
  determined up to renumbering and are called the {\em
    connected components} of $M$. 
\end{remark}

The {\em closure} of $A$ is the union of all
$X\subseteq E$  such that $X$ contains $A$ and the ranks of $X$ and
$A$ coincide. We denote by $\LL(M)$ the set of all closed sets, also
called {\em flats} of $M$. It is customary to endow $\LL(M)$ with the partial order given by inclusion, which makes it a geometric lattice, referred to as the {\em lattice of flats} of $M$.

 If the rank of $A\subseteq E$ equals its cardinality, we call $A$
 {\em independent}. Let $M_1$ and $M_2$ be matroids on the ground set
 $E$. Then $M_2$ is called a {\em weak map image} of $M_1$ if there is
 a bijection $\varphi: E\to E$ for which the preimage of every
 independent set of $M_2$ is an independent set of $M_1$.


\subsection{Matroid polytopes}\label{def:matroidpoly}

Let $M$ be a matroid of rank $d$ on the ground set
$[n]:=\{1,\ldots,n\}$. 
%
For every $B\in \BB(M)$ we consider a vector $v(B)\in \mathbb R^n$ defined by
  $v(B)_i=1$ if $i\in B$, $v(B)_i=0$ else. 
%
  The {\em matroid polytope} of $M$ is the convex hull
$$P_M:=\conv 
\{v(B) :  B\in \BB(M)\}.
$$
This is a polytope of dimension $n-c(M)$, where $c(M)$ is the number
 of connected components of $M$, 
contained in the hypersimplex
$\Delta_d=\conv\{de_i :  i=1,\ldots,n\}$. It can be readily seen that for two matroids $M_1,M_2$ we have
$P_{M_1\oplus M_2}=P_{M_1}\times P_{M_2}$.

Given a flat $F\in \LL(M)$, define the halfspace
$$
H^+_F:=\{x\in \mathbb R^n  :  \sum_{i\in F} x_i \leq \rank(F)\}
$$
and let $H_F$ be the hyperplane bounding $H_F^+$. According to \cite{FS04}, we have
$$
P_M=\Delta_d\cap \bigcap_{F\in \LL(M)} H_F^+.
$$

Let us consider the poset $\FF(M)$ of faces (closed cells) of $P_M$ ordered by
inclusion (see
\cite[Definition 2.6]{Zi95}). Every $f\in
\FF(M)$ is the matroid polytope of a matroid $M_f$ 
called ``the matroid type'' of the face $f$, with set of bases
$$
\BB(M_f) = \{B\in \BB(M)  :  v(B) \in V(f)\},
$$
where $V(f)$ denotes the set of vertices of $f$. Notice that a matroid
type's ground set is always the full ground set $E$, while minors have
strictly smaller ground sets. In general, a matroid type is a direct
sum of minors.

A maximal element $f\in\FF(M)$ (a {\em facet} of $P_M$) can be of
one of two types:
\newcommand{\fl}[1]{F_{#1}}
\begin{itemize}
\item[(i)] $f$ lies on the boundary of $\Delta_d$,
\item[(ii)] $f$ meets the interior of $\Delta_d$.
\end{itemize}


\begin{remark}\label{rem:loop}\hfill
  \begin{itemize}
  \item[(i)] If $f$ is of type (i), then there is $j$ such that $x_j=0$ for
    all $x\in f$. 
    This means that $j$ is not contained in any basis of $M_f$,
    i.e., $j$ is a loop of $M_f$. Conversely, if $j$ is a loop of
    $M_f$, then $f$ is of type (i).

    \item[(ii)] Following
    \cite[Prop. 2.6]{FS04}, the facet $f$ is of type (ii) if and only if $f\subseteq H_F$ for an $F\in \LL(M)$ such that $M/F$ and
    $M[F]$ are both connected.  Such an $F$ is called a {\em flacet} of
    the matroid $M$, and is uniquely determined by $f$.
    We will write $\fl{f}$ for the flacet corresponding to the
    type-(ii) facet $f$. The bases lying on this facet are the bases of the matroid $M/{F_f}\oplus M[F_f]$.
%
\label{rem:first}
\end{itemize}
\end{remark}

\subsection{Bergman complexes}\label{sec:bergmandef}

Let $M$ be a connected matroid of rank $d$ on the ground set $[n]$.

We consider the polar dual $P_M^{\vee}$ of $P_M$ (see,
e.g., \cite[Definition 2.10]{Zi95}) and let $\FF^\vee(M)$ denote its
poset of faces. Duality determines canonical order-reversing bijections
$\FF(M)\stackrel{\vee}{\to}\FF^\vee(M) \stackrel{\vee}{\to}\FF(M)$. In
particular, every face $\alpha\in \FF^\vee(M)$ has a matroid type $M_\alpha:=M_{\alpha^\vee}$.

\begin{definition}
The {\em Bergman complex} of $M$ is the polyhedral subcomplex $\Gamma_M$ of
$P_M^{\vee}$ with set of faces
$$
\Gamma_M:=\{\alpha\in \FF^\vee(M)  :  M_\alpha \textrm{ loopfree}\},
$$
which we regard as a downwards closed subposet of $\FF^\vee(M)$.
\end{definition}

\begin{remark}
Bergman's original definition of this space \cite{B71} is set theoretical, while the polyhedral structure was first studied by Bieri and Groves \cite{BiGr}. Our definition follows Feichtner and Sturmfels' approach \cite{FS04} and describes  what  Ardila and Klivans \cite{AK} call \emph{coarse subdivision}.
\end{remark}

\begin{remark}\label{rem:vertices}
  According to Remark \ref{rem:loop}, the vertices of the Bergman
  complex are 
  exactly the faces of the form $f^\vee$ where $f$ is a facet of type
  (ii).
\end{remark}

Let $\gamma_1,\ldots,\gamma_k$ be the vertices of a face $\alpha\in
\Gamma_M$. The face $\alpha^\vee$ of $P_M$ is the
intersection of its adjacent facets $\gamma_1^\vee,\ldots , \gamma_k^\vee$. 
%
Therefore 
\begin{displaymath}
\{v(B) : B\in \BB(M_{\alpha})\} =V(\alpha^\vee) = \bigcap_{i=1}^k
V(\gamma_i^\vee)=\bigcap_{i=1}^k 
\{v(B) : B\in \BB(M_{\gamma_i}) \}, 
\end{displaymath}
and 
we can write
\begin{displaymath}
  \BB(M_\alpha)=
 \bigcap_{i=1}^k\BB(M_{\gamma_i})
.
\end{displaymath}

\begin{remark}
  \label{rem:simp}
  The face $\alpha\in \Gamma_M$ is simplicial if every proper subset
  of its vertices determines a proper subface. Equivalently, $\alpha$
  fails to be simplicial if and only if
$$
\BB(M_\alpha) = \bigcap_{\gamma\in U} \BB(M_\gamma)
$$
for a proper subset $U\subsetneq V(\alpha)$.
\end{remark}
\subsection{Lattice path matroids}

Let $p,q$ be lattice paths in the plane with common starting point (say at the origin
$(0,0)$) and ending point (say at a point $(m,r)$). We will assume that
$p$ never goes below $q$. We will write $p$ and $q$ as words
$$
p=p_1 \ldots p_{m+r}\quad\quad\quad q=q_1\ldots q_{m+r}
$$
where each letter is $N$ or $E$, signaling a step North $(0,1)$ or
East $(1,0)$.

\def\call{\mathcal}
\def\II{\call I}
\def\SS{\call S}
\newcommand{\PP}[1]{\lceil #1 \rfloor}

By $\PP{p,q}$ we will denote the set of lattice paths from $(0,0)$ to $(m,r)$ that never go above $p$ or below $q$.

For any $s\in \PP{p,q}$ write $s=s_1\ldots s_{m+r}$ and define
$$
B(s):=\{i :  s_i=N\}.
$$


\begin{lemma}[\!\!\cite{BdMN}]
The set $\{B(s) :  s\in \PP{p,q} \}$ is the set of bases
  of a matroid $M(p,q)$ on the ground set $[m+r]$.
\end{lemma}

\begin{definition}
  A {\em lattice path matroid} is any matroid of the form $M(p,q)$ for
  two lattice paths $p,q$ as above.
\end{definition}

\begin{remark}[{\!\!\cite[Theorem 3.6]{BdMN}}] 
\label{char:connected} A lattice path matroid $M(p,q)$ is connected if and only if the
paths $p$ and $q$ never touch except at $(0,0)$ and $(m,r)$.
\end{remark}

\begin{assumption}
  Unless otherwise stated, in the following we will consider only connected
  lattice path matroids.
\end{assumption}

In order to understand the faces of the Bergman complex of lattice path matroids, we will often use the fact that contractions and deletions of lattice path matroids are again lattice path matroids whose bounding paths can be constructed  directly from the bounding paths of the original matroid (see \cite[Section 3.1]{BdM06})

\subsection{Bergman complexes of lattice path matroids}
\label{sec:berglpm}

Let $M=M(p,q)$ be a connected lattice path matroid.
%
%
%
%
%
  Given $B\in \BB(M)$, here and in what follows we will write $p(B)$ for the corresponding
  lattice path. 
Let $\alpha$ be a face of
the Bergman complex of $M$. 
The corresponding path $p(B)$ of a basis $B\in M_\alpha$ is called path of $M_\alpha$ or just $M_\alpha$-path.  
A {\em node} of the matroid type $M_\alpha$ is any integer point of the lattice 
that is visited by every path $p(B)$ with $B\in \BB(M_\alpha)$.  

\begin{definition}[Fundamental flats, bays, land necks]\label{def:S}
We will say that a lattice point $(y_1,y_2)$ on the upper path $p$ is
a {\em bay} of 
$p$ if
  $p_{y_1+y_2}p_{y_1+y_2+1}=EN$. Similarly, a point $(z_1,z_2)$ on $q$ is a bay for $q$ if
  $q_{z_1+z_2}q_{z_1+z_2+1}=NE$. Let $U_p$, resp.\ $U_q$, be the set of bays of $p$,
  resp. of $q$.
 The {\em fundamental flats} of $M$ are sets of the form $\{ 1,\ldots
 , y_1+y_2 \}$ for an $(y_1,y_2)\in U_p$ or of the form
 $\{z_1+z_1+1,\ldots ,m+r\}$ for $(z_1,z_2)\in U_q$.
We point the reader to \cite[Section 5]{BdM06} for further detail on the properties of these sets -- here it is enough to recall that these are flats of the matroid $M$. 

We will say that $i\in [m+r]$ is a {\em
  land neck} of $M(p,q)$ if the endpoint of
$p_1\ldots p_{i+1}$ lies one unit North of the endpoint  of
$q_1\ldots q_i$. The set of land necks is $S(p,q)$.
\end{definition}

\begin{figure}[h]
  \centering
  \begin{tikzpicture}
    \filldraw[fill=gray!5]
    (0,0) -- (1,0) -- (1,0.5)  -- (2,0.5) -- (2,1.5)  -- (2.5,1.5) -- (2.5,2.5)  -- (4,2.5) --
    (4,4)
     -- (2,4)  -- (2,3)  -- (1.5,3)  -- (1.5,1.5)  -- (1,1.5)  --
     (1,1)  -- (0,1);
   \filldraw[color=gray!20,fill=gray!25] (1.5,0) -- (2,0) -- (0,2) -- (0,1.5);
    \draw (1.75,0) node[inner sep=1pt][anchor=north] {\small 4}; 
    \draw (0,1.75) node[inner sep=1pt][anchor=east] {\small 4}; 
   \filldraw[color=gray!20,fill=gray!25] (3,0) -- (3.5,0) -- (0,3.5) -- (0,3);
    \draw (3.25,0) node[inner sep=1pt][anchor=north] {\small 7}; 
    \draw (0,3.25) node[inner sep=1pt][anchor=east] {\small 7}; 
    \draw (0,0) grid [step=0.5cm] (1,1); 
    \draw (1,0.5) grid [step=0.5cm] (2,1.5); 
    \draw (1.5,1.5) grid [step=0.5cm] (2.5,3);
    \draw (2,2.5) grid [step=0.5cm] (4,4); 
    \draw (1,1) -- (1,1.5);
    \draw (1.5,1.5) -- (1.5,3); 
    \draw (2,1.5) -- (2.5,1.5); 
    \draw (2,3) -- (2,4); 
    \draw (2.5,2.5) -- (4,2.5); 
    \draw [ultra thick] (1,.5) -- (1,1); 
    \node [fill=black,shape=circle,inner sep=2.5pt,] at (1.5,1.5){}; 
    \node [fill=black,shape=circle,inner sep=2.5pt] at (2,3){}; 
    \node [fill=black,shape=circle,inner sep=2.5pt] at (1,1){}; 
    \node [fill=black,shape=rectangle,inner sep=3pt] at (1,0.5){}; 
    \node [fill=black,shape=rectangle,inner sep=3pt] at (2,1.5){}; 
    \node [fill=black,shape=rectangle,inner sep=3pt] at (2.5,2.5){}; 
%
%
%
%
  \end{tikzpicture}
\caption{A lattice path matroid of rank $8$ on the ground set
  $[16]$. The black dots
mark the $p$-bays, the black squares mark the $q$-bays and the thick
line shows that the singleton $\{4\}$ is a land neck (however, note
that the singleton $\{7\}$ is not).}\label{fig:defs}
\end{figure}
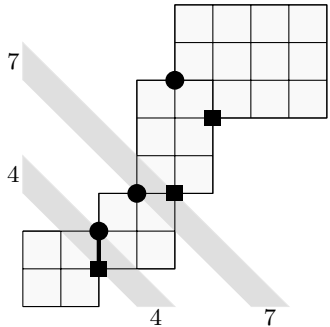

The following  lemma identifies vertices of the Bergman complex of a lattice path matroid.

\begin{lemma}
\label{rem:strait}
  The flacets (vertices of the Bergman complex) of a connected lattice path matroid $M(p,q)$ are
  \begin{itemize}
  \item[(a)]  the fundamental flats, 
  \item[(b)] the singletons $\{i\}\subseteq [m+r]\setminus S(p,q)$.
  \end{itemize}
\end{lemma}

\begin{proof}
According to Remark \ref{rem:vertices} and Remark \ref{rem:loop}, we need to identify the connected flats of $M(p,q)$ whose contraction is also connected. To this end, recall  that contractions and deletions of lattice path matroids are again lattice path matroids. More precisely, e.g.\ following \cite[Section 3]{BdM06},  for every $\{i\}\in [m+r]$ a representation of the contraction of $\{i\}$ in $M(p,q)$ is obtained as follows. If $i$ is not a loop nor a coloop, remove the last North step of $p$ before or at $p_i$ as well as the first North step of $q$ after or at $q_i$. If $i$ is a loop (resp.\ a coloop), then remove the common East (resp.\ North) step at $p_i = q_i$.
	With this we can already see  that part (b) of the claim identifies exactly the one-element connected flats whose contraction is connected (see  Figure \ref{fig:defs} for an illustration of this definition). 
		
By Theorem 5.7 of \cite{BdM06}, the nontrivial connected flats of $M(p,q)$ are 
\begin{itemize}
\item[(1)] the fundamental flats and 
\item[(2)] the flats of the form $F_y \cap G_z$, where $F_y$ (resp.\ $G_z)$ is the fundamental flat corresponding to some $y=(y_1,y_2)\in U_p$ (resp.\ some $z=(z_1,z_2)\in U_q$), and for which $y_1 > z_1$ (see e.g.\ Figure \ref{figfg}).\\
Moreover, in this case we have $\rk(F_y \cap G_z)=y_2-z_2$. 
\end{itemize}
We are thus left with showing that contractions of fundamental flats are connected, while contractions of the flats described in (2) are not.

First, consider a fundamental flat, say $F_y$ for some $y=(y_1,y_2)\in U_p$. By Remark \ref{def:CD}, the bases of the contraction to $F_y$ are of the form $B\setminus F_y$ where $B$ is a basis of $M$  whose intersection with $F_y$ has maximal rank, i.e., $\vert B\cap \{1,\ldots , y_1+y_2\}\vert=y_2$. These are the bases represented by paths which reach height $y_2$ in the first $y_1+y_2$ steps. It is easy to see that such paths are exactly those passing through $y$. We conclude that a lattice path representation for $M(p,q)/F_y$ consists of all lattice paths starting at $y$ and contained in $\lceil p,q \rfloor$.  With Remark \ref{char:connected} we see that  the bounding paths of this representation never meet except at $y$ and $(m,r)$ (since $M(p,q)$ is connected and $y$ is a $p$-bay), implying connectedness of the contraction. See Figure \ref{figf} for an illustration of this situation. The case of fundamental flats corresponding to $q$-bays is treated analogously.

Now we turn to the connected flats of type $(2)$, and we will prove that their contraction cannot be connected. In fact, let $F_y$ and $G_z$ as in (2) and write $y=(y_1,y_2)$, $z=(z_1,z_2)$. By the same argument as above we see that bases of $M(p,q)/(F_y \cap G_z)$ correspond to paths with $y_2-z_2$ North steps between
$\{z_1+z_2+1,\ldots,y_1+y_2\}$, which are exactly the paths passing through both $y$ and $z$. Thus, bases of 
the contraction correspond to concatenations of any path in $\lceil p,q \rfloor$ ending in $z$ with any path in $\lceil p,q \rfloor$ beginning at $y$ (Figure \ref{figfg} illustrates this situation). With Remark \ref{rem:CoCoCo} we see that  $M(p,q)/(F_y \cap G_z)$ is not connected.
%
\end{proof}

\begin{figure}
\centering
\begin{subfigure}[b]{0.3\textwidth}
\begin{tikzpicture}
\filldraw[color=gray!20,fill=gray!25] (0,0) -- (1,0) -- (1,1) -- (0,1);
\filldraw[color=gray!20,fill=gray!25] (1,0.5) -- (1,1.5) -- (1.5,1.5) -- (01.5,0.5);
\filldraw[color=gray!20,fill=gray!25] (1.5,0.5) -- (1.5,3) -- (2,3) -- (2,0.5);
\filldraw[color=gray!20,fill=gray!25] (2,3) -- (2.5,2.5) -- (2.5,1.5) -- (1.5,1.5);
\draw (0,0) grid [step=0.5cm] (1,1);
\draw (1,0.5) grid [step=0.5cm] (2,1.5);
\draw (1.5,1.5) grid [step=0.5cm] (2.5,3);
\draw (2,2.5) grid [step=0.5cm] (4,4);
\draw (1,1) -- (1,1.5);
\draw (1.5,1.5) -- (1.5,3);
\draw (2,1.5) -- (2.5,1.5);
\draw (2,3) -- (2,4);
\draw (2.5,2.5) -- (4,2.5);
\node [fill=black,shape=circle,inner sep=2pt, label=above left:{$y$}] at (2,3) {};
\end{tikzpicture}
\end{subfigure}
\begin{subfigure}[b]{0.3\textwidth}
\begin{tikzpicture}
\draw (2,3) grid [step=0.5cm] (4,4);
\draw (2,3) -- (2,4);
\draw (2,3) -- (4,3);
\node [fill=black,shape=circle,inner sep=2pt] at (2,3){};
\node [fill=white,shape=circle,inner sep=2pt] at (0,0){};
\end{tikzpicture}

\end{subfigure}
\caption{Illustration for the proof Lemma \ref{rem:strait}. Left-hand side: a fundamental flat $F_y=\{1,\ldots,y_1 + y_2\}$ corresponding to a $p$-bay $z=(y_1,y_2)$. Right-hand side:  its contraction $M/ F_y$.} 
\label{figf}
\vspace{0.5cm}

\begin{subfigure}[b]{0.3\textwidth}
\begin{tikzpicture}
\filldraw[color=gray!20,fill=gray!25] (1,0.5) -- (1,1) -- (0.5,1);
\filldraw[color=gray!20,fill=gray!25] (1,0.5) -- (1,1.5) -- (1.5,1.5) -- (1.5,0.5);
\filldraw[color=gray!20,fill=gray!25] (1.5,0.5) -- (1.5,3) -- (2,3) -- (2,0.5);
\filldraw[color=gray!20,fill=gray!25] (2,3) -- (2.5,2.5) -- (2.5,1.5) -- (1.5,1.5);
\draw (0,0) grid [step=0.5cm] (1,1);
\draw (1,0.5) grid [step=0.5cm] (2,1.5);
\draw (1.5,1.5) grid [step=0.5cm] (2.5,3);

\draw (2,2.5) grid [step=0.5cm] (4,4);
\draw (1,1) -- (1,1.5);
\draw (1.5,1.5) -- (1.5,3);
\draw (2,1.5) -- (2.5,1.5);
\draw (2,3) -- (2,4);
\draw (2.5,2.5) -- (4,2.5);
\node [fill=black,shape=circle,inner sep=2pt,label=below right:{$z$}] at (1,0.5){};
\node [fill=black,shape=circle,inner sep=2pt,label=above left:{$y$}] at (2,3){};
\end{tikzpicture}
\end{subfigure}
\hspace{2cm}
\begin{subfigure}[b]{0.3\textwidth}
\begin{tikzpicture}
\draw (1.5,1.5) grid [step=0.5cm] (3.5,2.5);
\draw (0.5,1) grid [step=0.5cm] (1.5,1.5);
\draw (0.5,1) -- (1.5,1);
\draw (1.5,1) -- (1.5,2.5);
\draw (1.5,1.5) -- (3.5,1.5);
\node [fill=black,shape=circle,inner sep=2pt] at (1.5,1.5){};
\node [fill=white,shape=circle,inner sep=2pt] at (0,0){};
\end{tikzpicture}
\end{subfigure}
\caption{Illustration for the proof Lemma \ref{rem:strait}. Left-hand side: a proper, non-trivial, connected flat $F_y\cap~G_z=\{z_1+z_2+1,\ldots,y_1+y_2\}$ of type (2). Right-hand side: its contraction $M / (F_y\cap~G_z)$.}
\label{figfg}
\end{figure}

Through Lemma \ref{rem:strait} we obtain 
 the following geometric interpretation of
Remarks \ref{rem:loop} and \ref{rem:vertices}.


\begin{remark}[Faces of the Bergman complex in terms of the lattice paths]\label{rem:FBCLP}
  Let $F$ be a flacet of $M=M(p,q)$ corresponding to a facet $f$ of the
  matroid polytope $P_M$ and consider
  $B\in \BB(M)$. 
  Then $B\in \BB(M_f)$ 
  if and only if one of the following holds
  \begin{itemize}
  \item[(a)] $F$ is a fundamental flat  and $p(B)$ goes through the corresponding bay,
  \item[(b)] $F=\{i\}$ and 
    $p(B)_i=N$.
  \end{itemize}
  We will then say that the path $p(B)$ {\em satisfies the constraint}
  imposed by $F$. 
In both cases we can describe the matroid type corresponding to a single vertex $F_f$ of the Bergman complex by 
\begin{displaymath}
M_f=M/{F_f} \oplus M[F_f].
\end{displaymath}

%
\label{interpret}
More generally, using this language of `paths' and `constraints' we can say that
 faces of the Bergman complex correspond to (the matroid type defined
 by paths satisfying the constraints given by) collections of
 fundamental flats and non-land-necks. Conversely, any such collection
 corresponds to a face of the Bergman complex provided that the paths
 satisfying its constraints determine a loopfree matroid.

With Remark \ref{rem:simp} we can also say that a face of the Bergman
complex of a lattice path matroid  is simplicial if and only if none
of the contraints associated to it is redundant (i.e., by removing any
of those we obtain a strictly smaller face).
\end{remark}

\begin{example}
On the left of
Figure \ref{fig:flacets} we have marked some of the flacets of the
lattice path matroid of Figure \ref{fig:defs}. 
The associated fundamental flats correspond to the $p$-bays $(3,3), (4,6)$
and the $q$-bay $(2,1)$. The paths that pass through these three
points and go North at their $8$th step   
define the bases of a matroid $M'$, of which we give a lattice path
representation on the right hand side
(we remark that in order to get such a presentation of the matroid
type, one has to change the initial order of the ground set). We 
see that $M'$ is loopfree, thus it is the matroid type of a face
$f$ of the Bergman complex, and that $\{8\}$ is the only singleton
which is a vertex of $f$. 
\begin{figure}[h]
\centering
\begin{subfigure}[b]{0.3\textwidth}
\begin{tikzpicture}
\filldraw[color=gray!20,fill=gray!25] (4,0) -- (3.5,0) -- (0,3.5) -- (0,4);
\draw (0,0) grid [step=0.5cm] (1,1);
\draw (1,0.5) grid [step=0.5cm] (2,1.5);
\draw (1.5,1.5) grid [step=0.5cm] (2.5,3);
\draw (2,2.5) grid [step=0.5cm] (4,4);
\draw (1,1) -- (1,1.5);
\draw (1.5,1.5) -- (1.5,3);
\draw (2,1.5) -- (2.5,1.5);
\draw (2,3) -- (2,4);
\draw (2.5,2.5) -- (4,2.5);
\draw (3.75,0) node[inner sep=1pt][anchor=north] {\small 8};
\draw (0,3.75) node[inner sep=1pt][anchor=east] {\small 8};
\node [fill=black,shape=circle,inner sep=2pt,] at (1.5,1.5){};
\node [fill=black,shape=circle,inner sep=2pt] at (2,3){};
\node [fill=black,shape=circle,inner sep=2pt] at (1,0.5){};
\node [] at (1,1.7){\small($3,3$)};
\node [] at (1.5,3.2){\small($4,6$)};
\node [] at (1.5,0.3){\small($2,1$)};
\end{tikzpicture}
\end{subfigure}
\hspace{3cm}
\begin{subfigure}[b]{0.3\textwidth}
\begin{tikzpicture}
\draw (0,0) grid [step=0.5cm] (1,0.5);
\draw (1,0.5) grid [step=0.5cm] (1.5  , 1.5);
\draw (1.5,1.5) grid [step=0.5cm] (2,2.5);
\draw (2,2.5) grid [step=0.5cm] (4,3.5);
\draw (1,0.5) -- (1,1.5);
\draw (1.5,1.5) -- (2,1.5);
\draw (1.5,1.5) -- (2,1.5);
\draw (1.5,1.5) -- (1.5,2.5);
\draw (2,2.5) -- (2,3.5);
\draw (2,2.5) -- (4,2.5);
\draw (4,3.5) -- (4,4);
\node [] at (4.2,3.75){\small 8};
\node [] at (1.35,1.8){\small 7};
\node [] at (1.35,2.3){\small 9};
\end{tikzpicture}
\end{subfigure}
\caption{Some flacets and a lattice representation of the matroid type
  $M_f$, where$f^\vee$ has those flacets as vertices.}
\label{fig:flacets}
\end{figure}
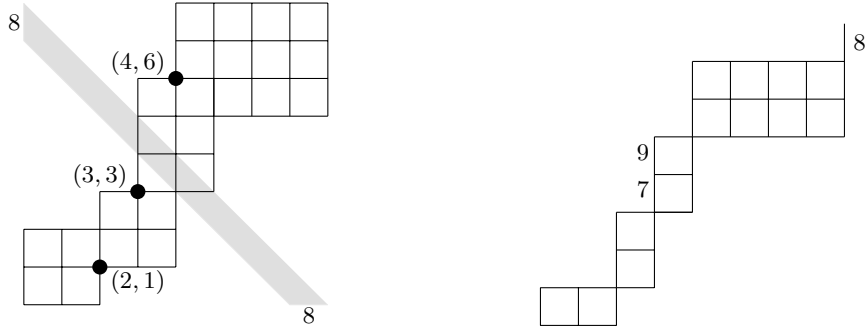
On the other hand, the $p$-bay $(3,3)$ and the $q$-bay $(4,3)$ do not
define a face of the Bergman complex, since every path that goes
through both will go East on the $7$th step, thus $\{7\}$ will be a
loop of the corresponding matroid type. Analogously, the $p$-bay $(3,3)$
and the two singletons $\{5\},\{6\}$ do not define a face of the
Bergman complex because every path that goes through $(3,3)$ and goes
North at the $5$th and $6$th step must go East (along $q$) on the
$4$th step, thus $\{4\}$ will be a loop of the corresponding matroid type.
\end{example}

\section{A simpliciality criterion}\label{sec:simpl}

The goal of this section is to give a complete characterization of
which lattice path matroids possess a simplicial Bergman
complex. After the discussions of Section \ref{sec:def}, it is clear
that it will be enough to consider connected (and in particular
loopfree) matroids.

\begin{theorem}
 \label{maintheo1}
Let $\alpha$ be a face of the Bergman complex of a connected lattice
path matroid $M(p,q)$. Then $\alpha$ is simplicial unless the flacets
corresponding to its vertices include
\begin{itemize}
\item[(1)] a fundamental flat $F$ corresponding to a bay $(x,z)$ of
  $p$ and
\item[(2)] a fundamental flat $G$ corresponding to a bay $(x,y)$ of
  $q$
\end{itemize}
with $z-y>1$.
\end{theorem}
\begin{proof}
Let the vertices of $\alpha$ include $F$ and $G$ as in (1) and (2) above. 
Then, for all $B\in\BB(M_\alpha)$, the lattice path $p(B)$ passes through $(x,y)$ and $(x,z)$. 
So it satisfies $p(B)_j=N$ for $j=x+y+1,\ldots, x+z$.
We conclude that if $F$ and $G$ correspond to vertices of $ V(\alpha)$ then all flacets 
in $\{x+y+1\},\ldots,\{x+z\}$ are forced by $F$ and $G$ to be vertices in $V(\alpha)$. More precisely, with $U:=V(\alpha)\setminus \{\{x+y+1\},\ldots,\{x+z\}\}$, we have 
$\BB(M_\alpha) = \bigcap_{\gamma\in U} \BB(M_\gamma)$.
 Since $z-y>1$, $U$ is a proper subset of $V(\alpha)$ and, in view of Remark \ref{rem:simp}%
, we conclude that $\alpha $ is not simplicial.

For the reverse implication, suppose $\alpha$ is not 
simplicial. Again by Remark \ref{rem:simp} there is $\gamma_0\in V(\alpha)$
such that
  \begin{equation}
  \BB(M_\alpha)=
  \bigcap_{\gamma\in U\setminus \{\gamma_0\}} \BB(M_\gamma)
  .\label{eq:2}
\end{equation}
We distinguish two cases.
\renewcommand{\descriptionlabel}[1]{\hspace{\labelsep}\emph{#1}}

\noindent {\em \underline{Case 1}: $F_{\gamma_0}$ is a fundamental flat.}
Without loss of generality $F_{\gamma_0}$ corresponds to a $p$-bay
$(x,z)$. In terms of Remark \ref{interpret}, Equation
\eqref{eq:2} means that every path satisfying the constraints of
$V(\alpha)\setminus \{\gamma_0\}$ passes through $(x,z)$. In particular $(x,z)$ is a node of $M_\alpha$. 
There may of course be other nodes south of $(x,z)$: in fact, the next claim proves that the southernmost node on the vertical of $(x,y)$ lies on the lower bounding path $q$.


\begin{description}
\item[Claim 1:] {\em Let $u$ be minimal such that $(x,u)$ is a node of $M_\alpha$. Then $(x,u)$ lies on $q$.}

\item[Proof of Claim 1.] 
First remark that any element of $\{ (x,u),\ldots,(x,z) \}$ is indeed a node of $M_\alpha$
Fix any path representing a basis of $M_\alpha$, say
$$
s=s_1\ldots s_{x+u-1} N \ldots N E s_{k} \ldots s_{r+m},
$$
where $k> x+z+1$.
Now consider the path
$$
s'=s_1\ldots s_{x+u-1} E N \ldots N  s_{k} \ldots s_{r+m}.
$$

If $s'=p(B)$ for a basis $B\in \BB(M)$, then it would satisfy the
constraints given by $F_\gamma$ for $\gamma\in V(\alpha)\setminus
\gamma_0$, since it passes all the same bays as $s$, and the only step where $s$ goes north but $s'$ doesn't is at $\{x+u\}$, which is not a vertex of $M_\alpha$ by choice of $u$ (otherwise $(x,u-1)$ would also be a node). However, $s'$ does not satisfy the constraint given by $F_{\gamma_0}$ since it does not pass 
through the $p$-bay $(x,z)$. Thus, the basis $B$ with $s'=p(B)$ would be an element in the right hand side but
not in the left hand side of Equation \eqref{eq:2}: a contradiction. 

We conclude that $s'\not\in\PP{p,q}$,
implying that the point $(x+1,u-1)$ is not a lattice point between $q$
and $p$, i.e., $(x,u)$ lies on $q$. \hfill $\square$
 \end{description}

Claim 1 shows in particular that there is a node $(x,u)$ of $M_\alpha$
 which lies on $q$. 
 If there is no $q$-bay of the form $(x,y)$, all paths of $M_\alpha$ go east in the $(x+u)$-th step (the step before reaching $(x,u)$). 
 But this is a contradiction to $M_\alpha$ being loopless. 
So there has to be a $q$-bay $(x,y)$ which is also a node of
$M_\alpha$. Now, since $M$ is connected, $y<z$; moreover, in order for
all paths satisfying the constraints of
$V(\alpha)\setminus\{\gamma_0\}$ to pass through $(x,z)$, $V(\alpha)$
must contain all singletons $\{x+y\},\ldots ,\{x+z\}$. In particular,
none of these is a land neck, thus $z-y>1$, q.e.d.


\noindent {\em \underline{Case 2}: $F_{\gamma_0}$ is a singleton $\{i'\}$.} Then 
  every path $s$ of $M_\alpha$ has $s_{i'}=N$, and we start by proving that if every path of $M_\alpha$ must go north
at step $i'$, then they all do so at a  node.

  \begin{description}
\item[Claim 2:] {\em $M_\alpha$ has a node $(x_1,x_2)$
      with $x_1+x_2=i'$.}

\item[Proof of Claim 2.] To see this, consider two paths $s,s'$ of $M_\alpha$ such that
    $s_1\ldots s_{i'}$ ends at $(y_1,y_2)$ and $s'_1\ldots s'_{i'}$
    ends at $(y_1',y_2')$ with $y'_2\geq y_2$. We want to prove that
    $y_2'=y_2$. 

    By way of contradiction suppose $y_2'-y_2 >0$ and let $(a_1,a_2)$,
    $(b_1,b_2)$ with $a_1+a_2<i'<b_1+b_2$ be on both $s,s'$ and such
    that the path $s_{a_1+a_2+1}\ldots s_{b_1+b_2}$ is always below
    $s'_{a_1+a_2+1}\ldots s'_{b_1+b_2}$. In particular,
    $s'_{a_1+a_2+1}=N$, $s'_{b_1+b_2}=E$ and there are no nodes of
    $M_\alpha$ between $(a_1,a_2)$ and $(b_1,b_2)$. The path $$
    s'_1\ldots s'_{a_1+a_2}N s_{a_1+a_2+2} \ldots s_{b_1+b_2-1} E
    s_{b_1+b_2+1}\ldots s_{r+m}$$ is thus a path of $M_\alpha$ that
    after $i'$ steps ends at a point
    $(y_1'',y_2'')$ with $y''_2=y_2+1$.\\
    By repeating this operation we can assume without loss of generality that $y_2'-y_2 = 1$. Now in this case, since the $i$th step of $s'$ is $N$, the
    path
$$
s'_1\ldots s'_{i'-1} E s_{i'+1} \ldots s_{r+m}
$$
represents an element of the right-hand side but not on the left-hand side of
\eqref{eq:2}: a contradiction. \hfill$\square$
\end{description}

Thus we see that all paths corresponding to bases of $M_\alpha$ must
pass through a common node, of the form $(x,i'-x)$, for some $x$.
Since $F_{\gamma_0}=\{i'\}$, before this node all paths must go
North. Thus, both $(x,i'-x-1)$ and $(x,i'-x)$
are nodes of $M_\alpha$.

Since $\{i'\}$ is not a land neck, we know that the points $(x,z)$ and
$(x,y)$ where the vertical line through $(x,i'-x)$ meets $p$ resp. $q$
must lie more than one unit apart, i.e., $z-y>1$.

To prove the theorem it now suffices to prove the following claim.

\begin{description}
\item[Claim 3:] {\em There is a $p$-bay $(x,z)$ north and a $q$-bay
    $(x,y)$ south of $(x,i'-x)$, and both the flacets
    corresponding to $(x,y)$ and $(x,z)$ are vertices of $\alpha$.}

\item[Proof of Claim 3.] Let $(x,u)$ be the lowest node of $M_\alpha$ south of $(x,i'-x-1)$.
Then there is a path $s$ of $M_{\alpha}$ with $s_{x+u}=E$.
Similarly let $v$ be maximal such that $(x,v)$ is a node of $M_{\alpha}$ and
 let $t$ be a path of $M_{\alpha}$ with $t_{x+v+1}=E$.
Consider the paths
\begin{align*}
  s' &:= s_1 \ldots s_{x+u-1}N \ldots N E s_{i'+1} \ldots s_{r+m}, \\
   t' &:= t_1 \ldots t_{i'-1} E N \ldots N t_{x+v+2} \ldots
  t_{r+m}.
\end{align*}

\begin{figure}[h]
  \centering
  \scalebox{0.8}{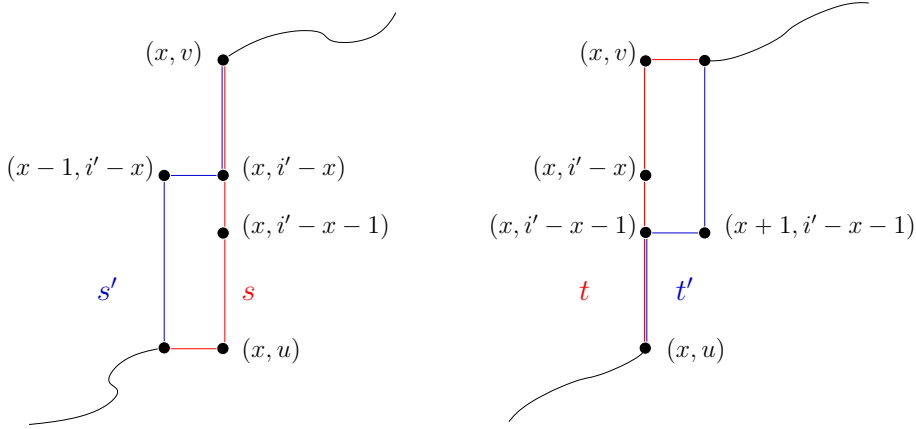}
  \caption{The paths $s$, $t$, $s'$ and $s$ in the proof of Claim 3.}
\end{figure}

Neither $s'$ nor $t'$ can be a path
of $M_\alpha$, because they violate
the constraint of $F_{\gamma_0}$ by $s'_{i'}=t'_{i'}=E$.
Therefore, in order for Equation \eqref{eq:2} to hold, for each of the paths $s'$
and $t'$ either
\begin{list}{}{}
\item[(i)] it is not an element of $\PP{p,q}$ at all, or 
\item[(ii)] it must violate the constraint
  given by some $F_\gamma$ with $\gamma\in V(\alpha)\setminus
  \{\gamma_0\}$.
\end{list}
Notice that in the second case, since the only step where $s'$ goes east but $s$ does not is $i'$,
the vertex $\gamma$ can not correspond to a singleton, and thus it
must correspond to a bay. In our particular setup we obtain the
following case analysis.

For $s'$ not to be a path of $M_\alpha$ one of the following must occur:
\begin{list}{}{}
\item[($s$i)] the point $(x,i'-x-1)$ lies on $p$ (since $s'$ passes through
  $(x-1,i'-x)$), or
\item[($s$ii)] there is $\gamma\in V(\alpha)$ such that $F_\gamma$
  corresponds to a $q$-bay $(x,y)$ between $(x,i'-x-1)$ and $(x,u)$. 
\end{list}

Similarly, for $t'$ not to be a path of $M_\alpha$, either
\begin{list}{}{}
\item[($t$i)] the point $(x,i'-x)$ lies on $q$ (since $t'$ passes through 
  $(x+1,i'-x-1)$), or
\item[($t$ii)] there is $\gamma\in V(\alpha)$ such that $F_\gamma$
  corresponds to a $p$-bay $(x,z)$ between $(x,i'-x)$ and $(x,v)$. 
\end{list}
Now, if both ($s$i) and ($t$i) were true, $M(p,q)$ would not be
connected, contradicting the hypothesis; on
the other hand, ($s$i) and ($t$ii) together imply that there is a
point on $p$ south of a $p$-bay, and similarly
from ($t$i) and ($s$ii) follows the existence of a point on $q$ that is
north of a $q$-bay. 
We conclude that both ($s$ii) and ($t$ii) must hold, proving the claim. 
\hfill$\square$
\end{description}

\end{proof}

\begin{corollary}\label{cor:iffs}
  The Bergman complex of a connected lattice path matroid $M(p,q)$ is
  simplicial if and only if every pair of vertically aligned bays determines a
  land neck (i.e., if $(x_1,x_2)\in U_p$ and $(x_1,x_2')\in U_q$, then
  $x_2-x_2' = 1$).
\end{corollary}

\begin{example}
 \label{ex2}
Consider the lattice path matroid given in Figure \ref{fig:ex}.
The face $\alpha$ whose vertices correspond to the bays $(4,3),(4,6)$
and the singletons $8,9,10$ in Figure \ref{fig:ex} is a
minimal non-simplicial face.
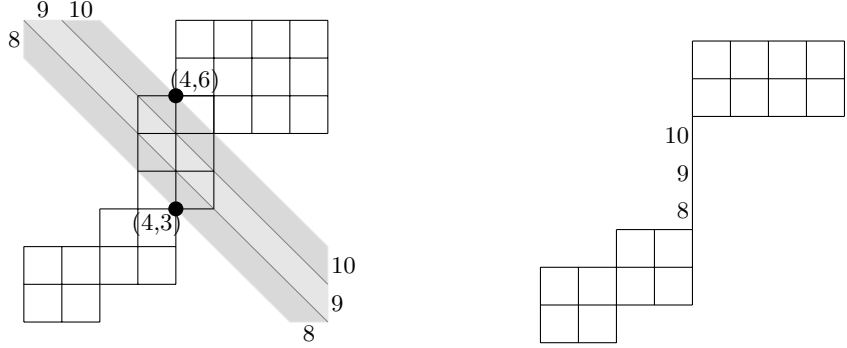
\begin{figure}[h]
\centering
\begin{subfigure}[b]{0.3\textwidth}
\begin{tikzpicture}
\filldraw[color=gray!20,fill=gray!20] (4,0) -- (4,0.5) -- (0.5,4) -- (0,4);
\filldraw[color=gray!20,fill=gray!30] (4,0) -- (3.5,0) -- (0,3.5) -- (0,4);
\filldraw[color=gray!20,fill=gray!30] (4,0.5) -- (4,1) -- (1,4) -- (0.5,4);
\draw [help lines] (4,0) -- (0,4);
\draw [help lines] (4,0.5) -- (0.5,4);
\draw (0,0) grid [step=0.5cm] (1,1);
\draw (1,0.5) grid [step=0.5cm] (2,1.5);
\draw (1.5,1.5) grid [step=0.5cm] (2.5,3);
\draw (2,2.5) grid [step=0.5cm] (4,4);
\draw (1,1) -- (1,1.5);
\draw (1.5,1.5) -- (1.5,3);
\draw (2,1.5) -- (2.5,1.5);
\draw (2,3) -- (2,4);
\draw (2.5,2.5) -- (4,2.5);
\draw (3.75,0) node[inner sep=1pt][anchor=north] {\small 8};
\draw (0,3.75) node[inner sep=1pt][anchor=east] {\small 8};
\draw (4,0.25) node[inner sep=1pt][anchor=west] {\small9};
\draw (0.25,4) node[inner sep=1pt][anchor=south] {\small9};
\draw (4,0.75) node[inner sep=1pt][anchor=west] {\small10};
\draw (0.75,4) node[inner sep=1pt][anchor=south] {\small10};
\node [fill=black,shape=circle,inner sep=2pt,] at (2,1.5){};
\node [fill=black,shape=circle,inner sep=2pt] at (2,3){};
\node [] at (1.75,1.3){\small(4,3)};
\node [] at (2.25,3.2){\small(4,6)};
\end{tikzpicture}
\end{subfigure}
\hspace{3cm}
\begin{subfigure}[b]{0.3\textwidth}
\begin{tikzpicture}
\draw (0,0) grid [step=0.5cm] (1,1);
\draw (1,0.5) grid [step=0.5cm] (2,1.5);
\draw (2,3) grid [step=0.5cm] (4,4);
\draw (1,1) -- (1,1.5);
\draw (2,1.5) -- (2,3);
\draw (2,3) -- (2,4);
\draw (2,3) -- (4,3);
\draw (2,1.75) node[inner sep=1pt][anchor=east] {\small 8};
\draw (2,2.25) node[inner sep=1pt][anchor=east] {\small 9};
\draw (2,2.75) node[inner sep=1pt][anchor=east] {\small 10};
\end{tikzpicture}
\end{subfigure}
\caption{Illustration of flacets of the Bergman complex that define a non-simplicial face as well as a lattice path representation of the matroid type of this face.}
\label{fig:ex}
\end{figure}
\end{example}

\section{Combinatorial structure}\label{sec:faces}


\subsection{The poset of faces}

Let $M=M(p,q)$ be a lattice path matroid with set of bays $U$ and set
of land necks $S(p,q)$. Define a partial order on $U$ by
setting
$$
(x_1,x_2) < (y_1,y_2) \textrm{ if and only if } x_1 \leq y_1
\textrm{ and }
x_2 < y_2.
$$
%
We will denote with
$\Delta(U)$ the set of chains (i.e., totally ordered subsets) of $U$, ordered by inclusion.
Every chain
$$
\omega:=\{ (a_1,b_1)< (a_2,b_2) < \cdots \}
$$
of $U$ defines a partition
\newcommand{\lang}[1]{\vert #1 \vert}
$$
\pi(\omega) = \pi_1(\omega) \uplus  \ldots \uplus \pi_{\lang{\omega}+1}(\omega)
$$
of the set $[m+r]$ with $j$-th block $\pi_j(\omega):=\{a_{j-1}+b_{j-1}+1 ,\ldots
,a_{j}+b_{j}\}$ for $j=1,\ldots, \lang{\omega}+1$, where we set $(a_0,b_0)=(0,0)$ and
$(a_{\lang{\omega}+1},b_{\lang{\omega}+1})=(m,r)$.

\begin{definition}
  The chain $\omega$ defines a matroid 
 $$
 M(\omega)=M_1(\omega) \oplus \ldots \oplus M_{\lang{\omega}+1}(\omega)
 $$
 where $M_j(\omega)$ is
 the lattice path matroid represented by the lattice paths from
 $(a_{j-1},b_{j-1})$ to $(a_j,b_j)$ lying between $p$ and $q$. 
\end{definition}

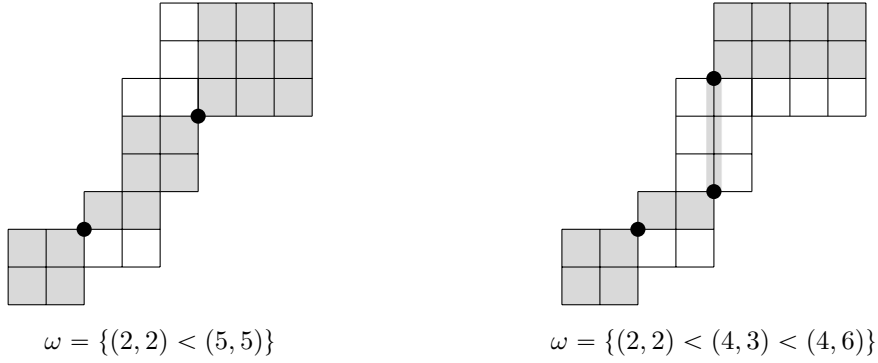
\begin{figure}[h]
\centering
\begin{subfigure}[b]{0.3\textwidth}
\begin{tikzpicture}
\filldraw[color=gray!20,fill=gray!30] (0,0) -- (1,0) -- (1,1) -- (0,1);
\filldraw[color=gray!20,fill=gray!30] (1,1) -- (2,1) -- (2,1.5) -- (1,1.5);
\filldraw[color=gray!20,fill=gray!30] (1.5,1.5) -- (2.5,1.5) -- (2.5,2.5) -- (1.5,2.5);
\filldraw[color=gray!20,fill=gray!30] (2.5,2.5) -- (4,2.5) -- (4,4) -- (2.5,4);
\draw (0,0) grid [step=0.5cm] (1,1);
\draw (1,0.5) grid [step=0.5cm] (2,1.5);
\draw (1.5,1.5) grid [step=0.5cm] (2.5,3);
\draw (2,2.5) grid [step=0.5cm] (4,4);
\draw (1,1) -- (1,1.5);
\draw (1.5,1.5) -- (1.5,3);
\draw (2,1.5) -- (2.5,1.5);
\draw (2,3) -- (2,4);
\draw (2.5,2.5) -- (4,2.5);
\node [fill=black,shape=circle,inner sep=2pt,] at (1,1){};
\node [fill=black,shape=circle,inner sep=2pt] at (2.5,2.5){};
\node [] at (2,-0.5){$\omega = \{(2,2)<(5,5)\}$};
\end{tikzpicture}
\end{subfigure}
\hspace{3cm}
\begin{subfigure}[b]{0.3\textwidth}
\begin{tikzpicture}
\filldraw[color=gray!20,fill=gray!30] (0,0) -- (1,0) -- (1,1) -- (0,1);
\filldraw[color=gray!20,fill=gray!30] (1,1) -- (1,1.5) -- (2,1.5) -- (2,1);
\filldraw[color=gray!20,fill=gray!30] (2,3) -- (2,4) -- (4,4) -- (4,3);
\filldraw[color=gray!20,fill=gray!30] (1.9,1.5) -- (2.1,1.5) -- (2.1,3) -- (1.9,3);
\draw (0,0) grid [step=0.5cm] (1,1);
\draw (1,0.5) grid [step=0.5cm] (2,1.5);
\draw (1.5,1.5) grid [step=0.5cm] (2.5,3);
\draw (2,2.5) grid [step=0.5cm] (4,4);
\draw (1,1) -- (1,1.5);
\draw (1.5,1.5) -- (1.5,3);
\draw (2,1.5) -- (2.5,1.5);
\draw (2,3) -- (2,4);
\draw (2.5,2.5) -- (4,2.5);
\node [fill=black,shape=circle,inner sep=2pt,] at (2,1.5){};
\node [fill=black,shape=circle,inner sep=2pt] at (2,3){};
\node [fill=black,shape=circle,inner sep=2pt] at (1,1){};
\node [] at (2,-0.5){$\omega = \{(2,2)<(4,3)<(4,6)\}$};
\end{tikzpicture}
\end{subfigure}
\caption{The direct summands for $M(\omega)$ are the lattice path
  matroids that appear in the shaded regions.
  }
\label{fig:minor}
\end{figure}


\begin{remark}
 The matroid $M(\omega)$ is a weak map image of $M$. 
  Also, notice that for $\omega\in \Delta(U)$ every $M_j(\omega)$ is loopfree. 
\end{remark}


Let the set 
$\mathcal I:= \mathscr P([m+r]\setminus S(p,q))$ 
be partially ordered by inclusion.

\def\QQ{\mathcal Q}
\begin{definition}\label{def:Q}  Given a lattice path matroid $M$ recall the
  above notations and
  define $\QQ(p,q)$ as the subposet of the product poset $\Delta(U)\times \mathcal I $ defined by:
  \begin{displaymath}
    \left\{(\omega, J)\in \Delta(U)\times \mathcal I~  \middle\vert
    \begin{array}{lr}
    \text{For }2\leq i \leq \lang{\omega}:   \\ 
	\pi_i(\omega) \cap J ~= ~ \pi_i(\omega)&\hspace{-2cm}\textrm{ if } a_i=a_{i-1}, b_{i}-b_{i-1}>1;\\
	\pi_i(\omega) \cap J ~= \emptyset &\hspace{-2cm}\textrm{ if } a_i=a_{i-1}, b_{i}-b_{i-1}=1;\\
        \pi_i(\omega) \cap J \textrm{ is an independent flat of } M_i(\omega) &\textrm{ otherwise}\\
      \end{array}
   \right\}.
  \end{displaymath}
\end{definition}



\begin{theorem}\label{maintheo2}
  For any connected lattice path matroid $M=M(p,q)$ the posets
 $
 \Gamma_{M(p,q)} $ and $ \QQ(p,q)
 $
 are isomorphic.
\end{theorem}

\begin{proof}
Recall that vertices of $\Gamma_M$ are the set of bays and non-land necks singletons. 
  Faces of $\Gamma_M$ are cells of $P_M^\vee$, and thus uniquely
  determined by their vertices.

We consider a face $\alpha$ of the Bergman complex of $M$ and recall
that, with Lemma \ref{rem:strait}, its set $V(\alpha)$ of vertices consists of two types of
elements: bays (identifying vertices corresponding to fundamental
flats) and singletons. Accordingly, we have a partition
$V(\alpha)=V(\alpha)\cap U \uplus V(\alpha)\cap \mathcal I$. We first
show that the pair  $(V(\alpha)\cap U, V(\alpha)\cap \mathcal I)$ is
contained in $\mathcal Q(p,q)$.
\begin{itemize}
\item[(a)] $V(\alpha)\cap U\in \Delta(U)$. \\ To prove this, consider two
 points $(a,b),\, (c,d)\in V(\alpha)\cap
  U$ such that $a\leq c$ and 
 $b\geq d$. Every lattice path $p(B)$ for  $B\in M_\alpha$ must pass through both points, therefore $a\leq c$
  implies $b\leq d$, so $b=d$.
 Now we have $p(B)_j=E$ for all $B\in \BB(M_\alpha)$ and every
  $j\in \{a+b +1,\ldots, c+b\}$ --- but since $M_\alpha$ by definition
  must be loopless, there can't be any such $j$, thus $a=c$ and we are
  done.
\end{itemize}
Write $\omega= (a_1,b_1) < \cdots < (a_k,b_k)$ for the chain
$V(\alpha)\cap U$, write $J:=V(\alpha)\cap \mathcal I$ and let $i\in [k]$.
\begin{itemize}
\item[(b)] If $a_{i-1}=a_i$, then either $b_i-b_{i-1}=1$ \emph{i.e.}
  the two bays define a land neck and
  $\pi_i(\omega)=\{a_i+b_i\}\subseteq S[p,q]$ and so $\emptyset =
  \pi_i(\omega)\cap \mathcal I \supseteq \pi_i(\omega)\cap J$, as
  required. Otherwise,  we are in the non-simplicial situation
  of the proof of Theorem \ref{maintheo1}, and every element of
  $\pi_i(\omega)=\{a_{i-1}+b_{i-1}+1,\ldots, a_i+b_i\}$ is, as a
  singleton flacet, a vertex of $\alpha$. Therefore, $\pi_i(\omega)
  \cap J = \pi_i(\omega)$.
\item[(c)] If $a_{i-1}<a_{i}$, then $\pi_i(\omega)\cap J  $ is
  an independent flat of $M$.\\ 
  Namely, in this case, thinking of paths satisfying constraints
  (according, e.g., to Remark \ref{rem:FBCLP}), the set of lattice paths
  $\{p(B)_{a_{i-1}+b_{i-1}+1}\cdots p(B)_{a_{i}+b_{i}} : B\in \BB( M_\alpha)\}$ 
  is the set of bases of a matroid $M'_i$  isomorphic to 
  $$M_i(\omega)/(J\cap \pi_i(\omega)) \oplus M_i(\omega)[J\cap \pi_i(\omega)].$$ 

Now, $\pi_i(\omega) \cap J$ is independent because every basis of $M_\alpha$ 
  contains $J$ (the paths representing these bases must satisfy the constraints in $J$, i.e., go north at every element of $J$). 
  The loop-freeness of $M_\alpha$ gives the loop-freeness of $M'_i$ (which is a direct summand in the decomposition of $M_\alpha$). In particular we see that $M_i(\omega)/(J\cap \pi_i(\omega))$ is loopfree. 
Then, $J\cap \pi_i(\omega)$ is a flat of $M_i(\omega)$ (e.g., by \cite[Exercise 3.1.8]{Ox11}). 
\end{itemize}

\noindent In view of (a), (b), (c) above, the following function is
well defined and clearly order-preserving.

\begin{equation*}
    \label{eq:1}
    \psi : \Gamma_M \to \QQ(p,q);\quad
                V \mapsto (V\cap U, V\cap \mathcal I)
  \end{equation*}
Moreover, it admits an order-preserving inverse given by
  \begin{equation*}
    \label{eq:23}
        \QQ(p,q) \to \Gamma_M; \quad        (\omega, J) \mapsto \omega\cup J,
  \end{equation*}
is easily seen to be well-defined, as it sends $(\omega,J)$ to the
vertex set of the face $\alpha$ with
$$
M_\alpha= 
\bigoplus_{a_{i-1}\neq a_i}
M_i(\omega)/(J\cap \pi_i(\omega))\oplus M_{i}(\omega)[J\cap \pi_i(\omega)] ~ \oplus \bigoplus_{a_{i-1}=a_i} M_i(\omega)
$$
which is loopfree because all its direct summands are.
\end{proof}

\subsection{Polyhedral structure of faces}

  As a face of $P^\vee$, every face of $\Gamma_M$ is the convex hull
  of its vertices. We would like to characterize such polyhedra.

%
    Recall Definition \ref{def:S} and, for  a subset $X\subseteq U\cup ([m+r]\setminus S(p,q))$, let
     $
\gamma(X)$
    denote the set of vertices $v$ of $P^\vee$ for which
     $F_v\in X$ either is an element of $X\setminus U$ or corresponds
     to a bay in $X\cap U$.

     \def\conv{\operatorname{conv}}
    Given a set of points $A$ in general position let $\Delta_{A}=\conv(A)$ denote the simplex on
    the vertex set $A$. Moreover, let $\lozenge_A$ be the polytope
    obtained as  the suspension of $\Delta_A$, see \cite[Section 2.2]{Kozlov} for a formal definition of suspensions. 

  \begin{remark}
    Recall that the poset of faces of the {\em join} $P\ast Q$ of two
    polytopes $P$ and $Q$ is obtained from the face posets
    $\widehat{\FF}(P)$ and $\widehat{\FF}(Q)$ (these are the face
    posets of $P$ and $Q$, each with an added unique smallest element
    $0_{P}$ resp. $0_{Q}$) by

\begin{displaymath}
  \widehat{\FF}(P\ast Q) \simeq \widehat{\FF}(P)\times \widehat{\FF}(Q).
\end{displaymath}

\end{remark}

\begin{corollary}\label{cor:poly}
  The face of $\Gamma_M$ corresponding
  to $(\omega,J)$ is a join
  \begin{displaymath}
  P(\omega,J):=\Delta_{\gamma(A)} \asterisk \left(\underset{i:~a_{i-1}<a_i}{\Asterisk}\Delta_{\gamma(J\cap \pi_i(\omega))}
 \right)\ast \left(\underset{\substack{i:~a_{i-1}=a_i,\\ b_i-b_{i-1}>1}}{\Asterisk}\lozenge_{\gamma(\pi_i(\omega))}
 \right)
  \end{displaymath}
  where $A:=\{(a_i,b_i)\in \omega :  a_{i-1}=a_i \Rightarrow \vert b_{i-1}-b_i\vert=0 \}$ is the
  {set} of all bays in $\omega$ that either form a land neck with some other bay of $\omega$ or are not vertically aligned with any other bay of $\omega$ at all. 
\end{corollary}

\begin{remark}
  Notice that if there are no two vertically aligned bays $(a_{i-1},b_{i-1}),(a_i,b_i)$ that do not form a land neck($a_{i-1}=a_i, b_i-b_{i-1}>1$), then every term
  of the above join is a simplex, and thus $P(\omega,J)$ is a
  simplex, in agreement with the condition found in Theorem \ref{maintheo1}.
\end{remark}

\begin{proof}
  Since the vertex set of $P(\omega,J)$ is precisely
  $\gamma(\omega\cup J)$, it is enough to prove
  isomorphism of face posets.

  Given a set $X$ and distinct elements $y_1,y_2\not\in X$, write $B_X$ for the poset of all subsets of $X$,
  $C_{X,y_1,y_2}$ for the poset of all subsets $Y\subseteq X\cup\{y_1,y_2\}$
  such that $\{y_1,y_2\}\subseteq Y$ if and only if $X\subseteq Y$. Then
  $B_X\simeq \widehat{\FF}(\Delta_{X})$ and $C_{X,y_1,y_2}\simeq
  \widehat{\FF}(\lozenge_X)$ and the face poset of $P(\omega,J)$ is isomorphic to
  \begin{displaymath}
    {B_{A}} \times \prod_{a_{i-1}< a_i} {B}_{J\cap \pi_i(\omega)} \times
\underset{\substack{i:~a_{i-1}=a_i,\\ b_i-b_{i-1}>1}}{\prod}
      C_{\pi_i(\omega),(a_{i-1},b_{i-1}),(a_i,b_i)}.
  \end{displaymath}
  Now to prove that the map $$\QQ(p,q)_{\leq (\omega,J)}\to
  \FF(P(\omega,J)),$$
  $$(\omega',J')\mapsto (\omega'\cap A,
  \underbrace{J'\cap\pi_i(\omega),\ldots}_{a_{i-1}<a_i},\underbrace{(\{(a_{i-1},b_{i-1}),(a_i,b_i)\}\cap\omega')
    \cup (J'\cap\pi_i(\omega)),\ldots}_{a_{i-1}=a_i, b_i-b_{i-1}>1} ) $$ is a poset isomorphism amounts to a
  routine check in view of Theorem \ref{maintheo1} and of the fact
  that, for every $(\omega,J)\in \QQ(p,q)$, we get that $(\omega,J')$ is in $\QQ(p,q)$ for
  all $J'\subseteq J$. Here, loop-freeness of the quotients
  $M_i(\omega)/(J'\cap\pi_i(\omega))$ is implied by loop-freeness of
  $M_i(\omega)/(J\cap\pi_i(\omega))$ since subsets of independent flats are again independent flats.
\end{proof}

\begin{example}
Consider the lattice path matroid given in Figure \ref{fig:ex} on page \pageref{fig:ex}.
The face $\alpha$ whose vertices correspond to the bays $(4,3),(4,6)$
and the singletons $8,9,10$ in Figure \ref{fig:ex} is a triangular
bipyramid obtained by suspending a $2$-simplex (whose vertices
correspond to the singletons $8$, $9$, $10$) between two additional
vertices (corresponding to the bays $(4,3)$ and $(4,6)$).
%
\end{example}

\bibliographystyle{plain}
\bibliography{lpmbib}{}

\end{document}